\documentclass{article}

\usepackage{arxiv}

\usepackage[utf8]{inputenc} 
\usepackage[T1]{fontenc}    
\usepackage{hyperref}       
\usepackage{url}            
\usepackage{booktabs}       
\usepackage{amsfonts}       
\usepackage{nicefrac}       
\usepackage{microtype}      
\usepackage{lipsum}		
\usepackage{graphicx}
\usepackage[numbers]{natbib}
\usepackage{doi}
\usepackage{amsmath}
\usepackage{amsthm}
\usepackage{orcidlink}
\newtheorem{lemma}{Lemma}
\newtheorem{theorem}{Theorem}

\title{An Integer Programming Formulation for the Maximally Diverse Grouping Problem}

\author{ Kevin Fu Yuan Lam \orcidlink{0000-0002-1535-7059}\\
	Office of the President\\
	National University of Singapore\\
	\texttt{lamfy@nus.edu.sg} \\
	\And
	Jiang Qian \orcidlink{0009-0004-7029-6879}\\
	Office of the President\\
	National University of Singapore\\
	\texttt{j.qian@nus.edu.sg} \\
}

\date{}

\renewcommand{\headeright}{}
\renewcommand{\undertitle}{}

\hypersetup{
pdftitle={A template for the arxiv style},
pdfsubject={q-bio.NC, q-bio.QM},
pdfauthor={David S.~Hippocampus, Elias D.~Striatum},
pdfkeywords={First keyword, Second keyword, More},
}

\begin{document}
\maketitle

\begin{abstract}
The Maximally Diverse Grouping Problem (MDGP) is the problem of assigning a set of elements to mutually disjoint groups in order to maximise the overall diversity between the elements. Because the MDGP is NP-complete, most studies have focused on heuristic solution approaches, as compared to exact solution approaches, to the problem. On the one hand, heuristic solution approaches, although common in practice, do not guarantee a global optimal solution. On the other hand, studies that have reformulated the problem as an integer linear programme, which can be solved using exact solution approaches, are either restricted to groups of equal size or restricted to the use of the Manhattan distance. The present paper presents a new integer linear programming formulation that is not subjected to either of these restrictions, and can therefore be used to establish useful benchmarks for the performance of heuristics in a broader range of applications moving forward.
\end{abstract}


\section{Introduction}

The Maximally Diverse Grouping Problem (MDGP) is the problem of assigning a set of $N$ elements to $G$ mutually disjoint groups in order to maximise the overall diversity between the elements \cite{ramos2020}. Since its initial application to assign students to groups \cite{weitz1992}, the MDGP has been argued to also have applications in organisations and scientific funding agencies \cite{fan2011}: in organisations, companies with more diverse management teams, with respect to ethnicity, race or gender, tend to report higher returns as compared to those with less diverse management teams \cite{rock2016}; and in scientific funding agencies, a diverse group of panellists, with respect to demographics and areas of specialisation, can ensure that multiple perspectives are represented during the review process \cite{hettich2006}.

The MDGP is often formulated as an Integer Quadratic Programme (IQP): Weitz and Lakshminarayanan \cite{weitz1998} formulated the MDGP as an IQP such that the group sizes must be the same for all pairs of groups; Palubeckis et al. \cite{palubeckis2011} formulated the MDGP as an IQP such that the group sizes can be different for some pairs of groups, but the minimum group sizes and the maximum group sizes must be the same for all pairs of groups; and Fan et al. \cite{fan2011} formulated the MDGP as an IQP such that the group sizes can be different for some pairs of groups, and the minimum group sizes and the maximum group sizes can also be different for some pairs of groups.

Because the MDGP is NP-complete \cite{feo1990}, most studies have focused on heuristic solution approaches, as compared to exact solution approaches \cite{schulz2022}, including but not limited to neighbourhood search algorithms, evolutionary algorithms and swarm intelligence algorithms \cite{ramos2020}. For example, Brimberg et al. \cite{brimberg2017} solved it using a skewed general variable neighbourhood search algorithm, Fan et al. \cite{fan2011} solved it using a hybrid genetic algorithm and Rodriguez et al. \cite{rodriguez2013} solved it using an artificial bee colony algorithm.

Although heuristic solution approaches are often used in practice, these approaches do not guarantee a global optimal solution \cite{karp1986}. Conversely, although exact solution approaches are not often used in practice, these approaches guarantee a global optimal solution, and therefore serve as useful benchmarks for the performance of heuristics \cite{papenberg2021}. As such, a few studies have reformulated the MDGP as an Integer Linear Programme (ILP), which can then be solved using exact solution approaches such as the branch-and-price algorithm, the branch-and-cut algorithm and the branch-and-bound algorithm (e.g. \cite{junger2009}): Papenberg and Klau \cite{papenberg2021} reformulated the Weitz and Lakshminarayanan \cite{weitz1998} formulation as an ILP; and Schulz \cite{schulz2022} reformulated the Fan et al. \cite{fan2011} formulation as an ILP.

However, both of these reformulations remain limited in their application. First, the Papenberg and Klau \cite{papenberg2021} reformulation, like the Weitz and Lakshminarayanan \cite{weitz1998} formulation, is restricted to groups of equal size. Second, although the Schulz \cite{schulz2022} reformulation, like the Fan et al. \cite{fan2011} formulation, is not restricted to groups of equal size, it is restricted to the use of the Manhattan distance.

In order to address both of these limitations, the present study seeks to reformulate the Palubeckis et al. \cite{palubeckis2011} formulation as an ILP. Because the Palubeckis et al. \cite{palubeckis2011} formulation is neither restricted to groups of equal size nor to the use of the Manhattan distance, the proposed reformulation can serve as a useful benchmark for the performance of heuristics aimed at a wider range of MDGP that are commonly encountered in practice.

\section{Integer Quadratic Programme}
\label{iqp}

Palubeckis et al. \cite{palubeckis2011} formulated the MDGP as an IQP. Suppose that there are $N$ elements labelled $i=1, \dots,N$ and $G$ groups labelled $g=1,\dots,G$. Let $a$ be the lower bound for the group size, $b$ be the upper bound for the group size and $d_{ij}$ be the distance between elements $i$ and $j$. Let $x_{ig}$ be an indicator variable that is $1$ if element $i$ is in group $g$ and $0$ otherwise. Then, the problem is to choose values of $x_{ig}$ so as to

\begin{eqnarray}
    \max && \Sigma_{g=1}^{G} \Sigma_{i=1}^{N-1} \Sigma_{j=i+1}^{N} d_{ij} x_{ig} x_{jg} \\
    \text{s.t.} && \Sigma_{g=1}^{G} x_{ig} = 1, \quad i=1,..,N \\
    && \Sigma_{i=1}^{N} x_{ig} \geq a, \quad g=1,...,G \\
    && \Sigma_{i=1}^{N} x_{ig} \leq b, \quad g=1,...,G \\
    \text{and} && x_{ig} \in \{0,1\}, \quad i=1,...,N, \quad g=1,...,G
\end{eqnarray}

Objective Function (1) maximises the distances between pairs of elements in the same group. Constraint (2) ensures that each element is in one and only one group. Constraint (3) ensures that each group has at least $a$ elements. Constraint (4) ensures that each group has at most $b$ elements. Constraint (5) ensures that the decision variables are binary.

\section{Integer Linear Programme}
\label{ilp}

Papenberg and Klau \cite{papenberg2021} formulated the MDGP as an ILP subject to the constraint that all group sizes are equal. In this section, we describe their formulation and propose a formulation not subject to the constraint that all group sizes are equal.

\subsection{Equal Sized Groups}
\label{ilp_equal}

Suppose that there are $N$ elements labelled $i=1,\dots,N$ and $G$ groups labelled $g=1,\dots,G$. Let $a$ be the lower bound for the group size, $b$ be the upper bound for the group size and $d_{ij}$ be the distance between elements $i$ and $j$. Let $x_{ij}$ be an indicator variable that is $1$ if elements $i$ and $j$ are in the same group and $0$ otherwise. The problem is to choose values of $x_{ij}$ so as to

\begin{eqnarray}
    \max && \Sigma_{i=1}^{N-1} \Sigma_{j=i+1}^{N} d_{ij} x_{ij} \\
    \text{s.t.} && x_{ij} + x_{jk} - x_{ij} \leq 1, \quad \forall \text{ } 1 \leq i < j < k \leq N \\ 
    && x_{ij} + x_{ik} - x_{jk} \leq 1, \quad \forall \text{ } 1 \leq i < j < k \leq N \\
    && x_{ik} + x_{jk} - x_{ij} \leq 1, \quad \forall \text{ } 1 \leq i < j < k \leq N \\
    && \Sigma_{1 \leq i < j \leq N} x_{ij} + \Sigma_{1 \leq j < i \leq N} x_{ji} = \tfrac{N}{G} - 1, \quad \forall \text{ } i \in \{1, ..., N\} \\
    \text{and} && x_{ij} \in \{0,1\}, \quad i=1,...,N-1, \quad j=i+1,...,N
\end{eqnarray}

Objective Function (6) maximises the distances between pairs of elements in the same group. Constraints (7)–(9) ensures that any two elements that are in the same group as some third element must also be in the same group as each other. Therefore, it ensures that each element is in at most one group. Constraint (10) ensures that each group has $\tfrac{N}{G}$ elements. Therefore, it ensures that each element is in at least one group. Constraint (11) ensures that the decision variables are binary.

\subsection{Unequal Sized Groups}
\label{ilp_unequal}

The Papenberg and Klau \cite{papenberg2021} formulation is limited because it is restricted to groups of equal sizes \cite{schulz2022}. Schulz \cite{schulz2022} suggests that that the Papenberg and Klau \cite{papenberg2021} formulation could be made equivalent to the Palubeckis et al. \cite{palubeckis2011} formulation if Constraint (10) is bounded between $a-1$ and $b-1$ inclusive. However, doing so does not make the Papenberg and Klau \cite{papenberg2021} formulation equivalent to the Palubeckis et al. \cite{palubeckis2011} formulation because Constraints (3) and (4) in the Palubeckis et al. \cite{palubeckis2011} formulation state not just that all group sizes are between $a$ and $b$ inclusive for any number of groups, but all group sizes are between $a$ and $b$ inclusive for $G$ number of groups. 

For example, suppose that there are six elements with values $1$, $2$, $3$, $4$, $5$ and $6$. Further, suppose that each group must have either $2$ or $3$ elements (i.e., $a=2$ and $b=3$), and that there must be $3$ groups (i.e., $G=3$). The optimal value is $9$ with optimal solutions such as $\{1,5\}$, $\{2,4\}$ and $\{3,6\}$. However, the Papenberg and Klau \cite{papenberg2021} formulation with Constraint (10) bounded between $a-1$ and $b-1$ inclusive returns an optimal value of $16$ with optimal solutions such as $\{1,3,6\}$ and $\{2,4,5\}$ which are infeasible given that $a=2$ and $G=3$.

In order to address the limitation, we build on the Papenberg and Klau \cite{papenberg2021} formulation, and the Schulz \cite{schulz2022} suggestion to bound Constraint (10) between $a-1$ and $b-1$ inclusive, to reformulate the Palubeckis et al. \cite{palubeckis2011} formulation of the MDGP through the addition of a group count constraint.

Suppose that there are $N$ elements labelled $i=1,\dots,N$ and $G$ groups labelled $g=1,\dots,G$. Let $a$ be the lower bound for the group size, $b$ be the upper bound for the group size and $d_{ij}$ be the distance between elements $i$ and $j$. Let $x_{ij}$ be an indicator variable that is $1$ if elements $i$ and $j$ are in the same group and $0$ otherwise. Let $y_{i}$ be an indicator variable that is $1$ if element $i$ is the smallest index in its group and $0$ otherwise. The problem is to choose values of $x_{ij}$ and $y_{i}$  so as to

\begin{eqnarray}
    \max && \Sigma_{i=1}^{N-1} \Sigma_{j=i+1}^{N} d_{ij} x_{ij} \\
    \text{s.t.} && x_{ij} + x_{jk} - x_{ij} \leq 1, \quad \forall \text{ } 1 \leq i < j < k \leq N \\ 
    && x_{ij} + x_{ik} - x_{jk} \leq 1, \quad \forall \text{ } 1 \leq i < j < k \leq N \\
    && x_{ik} + x_{jk} - x_{ij} \leq 1, \quad \forall \text{ } 1 \leq i < j < k \leq N \\
    && \Sigma_{1 \leq i < j \leq N} x_{ij} + \Sigma_{1 \leq j < i \leq N} x_{ji} \geq a - 1, \forall \text{ } i \in \{1,...,N\} \\
    && \Sigma_{1 \leq i < j \leq N} x_{ij} + \Sigma_{1 \leq j < i \leq N} x_{ji} \leq b - 1, \forall \text{ } i \in \{1,...,N\} \\
    && x_{ij} + y_{i} \leq 1, \forall \text{ } i \in \{1,...,N-1\}, \quad \forall \text{ } j \in \{i+1,...,N\} \\
    && \Sigma_{i=1}^{j-1} x_{ij} + y_{j} \geq 1, \forall \text{ } j \in \{2,...,N\} \\
    && \Sigma_{j=2}^{N} y_{j} = G - 1 \\
    \text{and} && x_{ij} \in \{0,1\}, \quad i=1,...,N-1, \quad j=i+1,...,N \\
    && y_{i} \in \{0,1\}, \quad i=2,...,N
\end{eqnarray}

Objective Function (12) maximises the distances between pairs of elements in the same group.  Constraints (13) – (15) ensures that any two elements that are in the same group as some third element must also be in the same group as each other. Therefore, it ensures that each element is in at most one group. Constraint (16) ensures that each group has at least $a$ elements. Therefore, it ensures that each element is in at least one group. Constraint (17) ensures that each group has at most $b$ elements. Constraints (18) – (20) ensure that there are $G$ groups. In particular, Constraint (18) ensures that an element $i$ is not the smallest index in its group if it is grouped with a smaller index, Constraint (19) ensures that an element $i$ is the smallest index in its group if it is not grouped with a smaller index and Constraint (20) ensures that the number of elements with the smallest indices in their groups is equal to the number of groups. Constraints (21) and (22) ensure that all decision variables are binary.

The following results prove the equivalence between the present ILP formulation and the original Palubeckis et al. \cite{palubeckis2011} IQP formulation.

\begin{lemma} 
If $x_{ij}$ satisfy constraints (13) – (15) and (21), then the index set $I={1,\dots,N}$ can be partitioned into several disjoint subsets $I_1,\dots,I_G'$ such that
\begin{eqnarray}
    I &=& \{1,\dots,N\} = I_{1} \cup \dots \cup I_{G'}, \quad I_{s} \cap I_{t} \neq 0, \quad \forall \text{ } s \neq t \\
    x_{ij} &=& 
    \begin{cases}
        1, \quad i,j \in I_{s} \\
        0, \quad i\in I_{s}, j \in I_{t}, s \neq t 
    \end{cases}
\end{eqnarray}
\end{lemma}

\begin{proof}
    If all $x_{ij}=0$, then the lemma follows directly with $I_{i}=\{i\}(i=1,\dots,N)$. Now assume $x_{ij}=1$ for some $i,j$. Set $I_{1}=\{i,j\}$. If $x_{ik}=x_{jk}=0$ for all $k \in I \backslash I_{1}$, set $I_{1}$ fixed, and continue with $I \backslash I_{1}$. Otherwise, there exists some $k$ such that $x_{ik}=1$ or $x_{jk}=1$. Assume $x_{ik}=1$, then it follows from (14) that $x_{jk}=1$. Update $I_{1} = I_{1} \cup \{k\}$ and continue. If there exists an index $l \in I \backslash I_{1}$ such that $x_{ml}=1$ for some $m \in I_{1}$, then it follows from (13) – (15) that $x_{ml}=1$ for all $m \in I_{1}$. We could update $I_{1} = I_{1} \cup \{l\}$. Repeat the process until $x_{ij}=0$ for all $i \in I_{1}$ and $j \in I \backslash I_{1}$. Then set $I_{1}$ fixed and continue with $I \backslash I_{1}$. Repeat the above process, until all indexes are assigned. The final partition will satisfy the results as in the lemma.
\end{proof}

Constraints (16) and (17) ensure that the number of elements in each group must lie between $a$ and $b$. We only remain to show that there will be at least $G$ groups with constraints (18), (20) and (22).

\begin{theorem}
    With constraints (18), (20) and (22), $G' \geq G$.
\end{theorem}

\begin{proof}
    We can show that if $i,j \in I_{s}$ for some $s$, then $y_{i}$ and $y_{j}$ cannot both equal to $1$. Indeed, if $y_{i}=y_{j}=1$, noting $x_{ij}=1$ since $i,j \in I_{s}$, $x_{ij}+y_{j}=2$ contradicts with (18). This means that in each subset $I_{s}$, $y_{j}$ can take the value of $1$ for at most one $j \in I_{s}$. As (20) tells that there will be exactly $G$ $y_{j}$-s being $1$, so the number of disjoint subsets $G'$ cannot be less than $G$.
\end{proof}

\begin{theorem}
    Additionally with the constraint (19), $G'=G$.
\end{theorem}

\begin{proof}
    Assume that there exists some $I_{s}$ such that $y_{j}=0$ for all $j \in I_{s}$. Taking the smallest index $j_{1}$ in $I_{s}$, then $x_{ij}=0$ for all $i \leq j_{1}$, and hence
    \begin{eqnarray}
        \Sigma_{i < j_{1}} x_{ij} + y_{j_{1}} = 0
    \end{eqnarray}
    
    contradicting (19). This implies that for each subset $I_{s}$, $y_{j}$ will take the value of $1$ for at least one $j \in I_{s}$. Hence $G' \leq G$, which, together with Theorem 1 shows that $G'=G$.
\end{proof}

\section{Conclusion}
\label{conclusion}

In the present paper, we have reformulated the Palubeckis et al. \cite{palubeckis2011} formulation of the MDGP as an ILP. The reformulation builds on previous reformulations and suggestions provided by Papenberg and Klau \cite{papenberg2021} and Schulz \cite{schulz2022}, and is neither restricted to groups of equal size nor to the use of the Manhattan distance. The relaxation of the constraint of groups of equal size is important, especially since common applications of the MDGP, including but not limited to educational institutions, business organisations and scientific funding agencies, require the assignment of individuals to groups that are not necessarily of the same size. In addition, because there are no restrictions on the distance metric, the distance between two elements can be measured using various approaches, including the Gower metric which is able to quantify distances using both quantitative (e.g., students' test scores) and qualitative (e.g., students' major) information \cite{gower1971}.

Nonetheless, the present paper is not without limitations. For example, the current reformulation of the Palubeckis et al. \cite{palubeckis2011} formulation of the MDGP remains a special case of the Fan et al. \cite{fan2011} formulation of the MDGP. In other words, although it is not restricted to groups of equal size, it still requires the minimum group size, as well as the maximum group size, to be the same for all groups. As discussed above, the Schulz \cite{schulz2022} reformulation requires neither the minimum group size nor the maximum group size to be the same for all groups, but is restricted to the use of the Manhattan distance. Future research can relax the corresponding constraints in either the current reformulation or the Schulz \cite{schulz2022} reformulation for even wider applicability.

In conclusion, although numerous heuristic solution approaches have been proposed, and are used in practice, these approaches to not guarantee a global optimal solution. Given that practical applications of the MDGP often encounter groups of unequal sizes, as well as use metrics other than the Manhattan distance, the present paper, which proposes a reformulation that relaxes both of these constraints, and that can be solved using exact solution approaches which guarantee a global optimal solution, enables us to establish useful benchmarks for the performance of these heuristics in a broader range of applications moving forward.

\vspace{0.5cm}
\textbf{Acknowledgements}

We would like to acknowledge and thank Gan Tian, Gao Xuru, Song Kaiyue, Wang Guangyu and Zhang Jiayue for helpful discussions during the initial phases of this project. 

\bibliography{references}  






\end{document}